\documentclass[12pt]{amsart}

\usepackage{templatestyle}
\usepackage{templatecommand}
\begin{document}

%titles
\title[Blow-up behavior of the scalar curvature along the CKRF]%top of the pages, abbreviate if neccesary
{Blow-up behavior of the scalar curvature along the conical K\"ahler-Ricci flow with finite time singularities} %title

\author%[Ryosuke Nomura] %top of the pages
{Ryosuke Nomura}%full name

\address{Graduate School of Mathematical Sciences, The University of Tokyo \endgraf
	3-8-1 Komaba, Meguro-ku, Tokyo, 153-8914, Japan.}

\email{nomu@ms.u-tokyo.ac.jp}

\thanks{Classification AMS 2010: 53C55, %Differential geometry/Hermitian and Kahlerian manifolds
	32W20.%Several complex variables and analytic spaces/Complex Monge-Ampere operators
}

\keywords{Conical K\"ahler-Ricci flow, twisted K\"ahler-Ricci flow, Monge-Amp\`ere equation, Cone metric, scalar curvature.
}

\begin{abstract}
	We investigate the scalar curvature behavior along the normalized conical K\"ahler-Ricci flow $\omega_t$, which is the conic version of the normalized K\"ahler-Ricci flow, with finite maximal existence time $T<\infty $. 
	We prove that the scalar curvature of $\omega_t$ is bounded from above by $C/(T-t)^2$ under the existence of a contraction associated to the limiting cohomology class $[\omega_T]$. 
	This generalizes Zhang's work to the conic case. 
\end{abstract}
	
\maketitle
\tableofcontents

\section{Introduction}

Let $X$ be a compact \kahler \ manifold of dimension $n$, $D$ be a smooth divisor on $X$, and $\beta $ be a positive real number satisfying $0<\beta<1$.
We consider the \nckrf \ $\omt$ on $X$ which is a family of cone metrics  with cone angle $2\pi \beta$ along $D$ satisfying the following evolution equation:
\begin{align}\label{nckrf}
	\begin{cases}
		\ \dt \omt &= - \rict - \omt + 2\pi ( 1 - \beta )[D], \\[7pt]
		\ \omt |_{t=0} &= \omstar,
	\end{cases}
\end{align} 
where $[D]$ is the current of integration over $D$, and $\omstar$ is a certain initial cone metric defined later (see (\refs{initialmetric})). 
In the case of $D=0$, $\omt $ is called the \nkrf. 
This case has been studied extensively in the past decades (see 
\cite{TianZhang06KRFProjGenType, SongTian09KRFthroughsing, SongTian12CanMeasKRF, ChenWang12SpaceRFI, ChenWang14SpaceRFII,BEG13IntrotoKRF, CheSunWang15KRFKEKstab, CollinsTosatti15KaCurrentNullLoci, GuoSongWeikove15GeometricConv} and the references therein). 

The maximal existence time $T$ of the \nckrf  \ is characterized by the following cohomological condition:
\begin{align*}
	T = \sup \{ t >0 \mid   [\omt ] =\emt [\omz] + (1-\emt) 2\pi c_1(K_X+(1-\beta )D)   \mbox{ \ is \kahler}\}, 
\end{align*}
which is shown by Shen~\cite{Shen14CKRF,Shen14C2alphaCKRF}. 
In particular, the limiting class $[\omT ]$ is nef but not \kahler.  
As $t$ tends to $T$, the flow $\omt$  forms singularities.  
The analysis of the singularities, in particular its curvature behavior, is one of the main objects in the study of the geometric flows. Our purpose here is to investigate the scalar curvature behavior of $\omt$ with finite time singularities (i.e. $T<\infty$) as $t$ approaches to $T$. 

In the infinite time singularities case (i.e. $T=\infty$), the uniform boundedness of the scalar curvature of the \nkrf \ (i.e. $D=0$) was 
proved by Zhang~\cite{Zhang09ScalKRFMinGenType} when the canonical bundle $K_X$ is nef and big. 
This result was extended by Song-Tian~\cite{SongTian11ScalKRF} when $K_X$ is semi-ample. 
Furthermore, Edwards~\cite{Edwards15ScalBoundCKRF} generalized these results to the conic setting.  
In the case of Fano manifolds, Perelman  (see \cite{SesumTian08Perelman}) established a uniform bound for  the scalar curvature along the  \nkrf \
 and Liu-Zhang~\cite{LiuZhang14CKRFFano} extended it to the conic case.

On the other hand, in the finite time singularities case (i.e. $T<\infty $), Collins and Tosatti~\cite{CollinsTosatti15KaCurrentNullLoci} proved that  the scalar curvature of the \krf \ $\omt$  (i.e. $D=0$) blows up along the null locus of the limiting class $[\omT]$. 
Zhang~\cite{Zhang10ScalFTSingKRF} showed that the scalar curvature $R(\omt)$ of the \nkrf \ $\omt$ satisfies
\begin{align*}
	R(\omt )\le \dfrac{C}{(T-t)^2}
\end{align*}
assuming the semi-ampleness of the limiting class $[\omega_T]$.
This condition is natural in terms of the deep relationship between the \krf \ and the minimal model program (see \cite{SongTian09KRFthroughsing,Zhang10ScalFTSingKRF}).
Our main theorem generalizes Zhang's result to the conic setting.

We assume the following contraction type condition on the limiting cohomology class $[\omega_T]$. 
Let $f \colon X \rightarrow Z$ be a holomorphic map between compact \kahler \ manifolds, whose image is contained in a normal irreducible subvariety $Y$ of $Z$.
Let $D_Y$ be an effective Cartier divisor on $Y$ such that the pullback of $D_Y$ satisfies $D = f^\ast D_Y$. 
Let $h_Y$ be a smooth Hermitian metric on the line bundle $\mathcal{O}_Y(D_Y)$ in the sense of  \cite[Section 5]{EGZ09SingKE}, 
and  $s_Y $ be a holomorphic section of $\mathcal{O}_Y(D_Y)$ whose zero divisor is $D_Y$.
We define the initial cone metric $\omstar$ by
\begin{align}\label{initialmetric}
	\omstar \deq \omz + k \ddbsb,
\end{align}
where $\omz$ is a smooth \kahler \ form on $X$, $k\in \rp$ is a sufficiently small real number, $s \deq f^\ast s_Y$ is the holomorphic section of $\mathcal{O}_X(D)$, and $h\deq f^\ast h_Y$ is the smooth Hermitian metric on $\mathcal{O}_X(D)$. We remark that if we take $k$ sufficiently small, $\omstar$ is actually a cone metric with cone angle $2\pi \beta$ along $D$.

Let $\omt $ be the \nckrf \ with initial cone metric $\omstar$, and $T$ be the maximal existence time of $\omt$. We further assume that $T$ is finite and there exists a smooth \kahler \ form $\omega_Z$ on $Z$ satisfying 
\begin{align*}
	[f^\ast \omega_Z] = [\omega_T] \in H^{1,1}(X, \rr ).
\end{align*}
Under these assumptions, we have the following theorem.
\begin{theorem}\label{thmA}
		The scalar curvature $R( \omt )$ of $\omt$ satisfies
	\begin{align*}
		R( \omt ) \le \frac{C}{(T-t)^2} \son X\setminus D,
	\end{align*}
		where $C>0$ is a constant independent of $t$.
\end{theorem}

In contrast with Zhang's result, we need to treat with the singularities of $\omt$ along $D$. This is overcome by using the approximation technique used in \cite{CampanaGuenanciaPaun13ConeSingHolTensor,Shen14CKRF,LiuZhang14CKRFFano,Edwards15ScalBoundCKRF}. 

\begin{remark}
		If we replace $(1-\beta )D$ by $\sum _{i \in I}(1-\beta_i)D_i$, where $D_i$ are smooth divisors intersecting transversely, the same argument below gives the same conclusion. But for simplicity, we only treat one smooth divisor case.
\end{remark}

\section{Approximation of the \nckrf \ by the twisted normalized K\"ahler-Ricci flow}
	
In the following argument, we assume that the conditions in Theorem \refs{thmA} are always satisfied. 
We first define a family of reference smooth \kahler \ forms $\omht$ whose cohomology classes are equal to $[\omt]$. 
We set $\omhinf $ by
\begin{align*}
	\omhinf 
		&\deq -\frac{\emT}{1-\emT } \omz + \frac{1}{1-\emT} f^\ast \omega_Z \\
		&\, \>  \in -\frac{\emT}{1-\emT } [\omz] +\frac{1}{1-\emT}[\omT ]= \ckd, 
\end{align*}
and $\omht$ by
\begin{align}\label{omht}
	\omht  \deq
		&\ \emt \omz + (1-\emt) \omhinf 
	=\ a_t \omz +(1-a_t) \omhT,
\end{align}
where $a_t \deq (\emt -\emT)/(1-\emT)$.
Since $\omhT = f^\ast \omega_Z\ge 0$ is semi-positive, $\omht$ are smooth \kahler \ forms for any $t \in [0,T)$. The cohomology class of $\omht$ coincide with $[\omt]$.

We next define a family of reference smooth \kahler \ forms $\omtilet$ whose cohomology classes are equal to $[\omt]$. We use the approximation method as in \cite{Shen14CKRF,LiuZhang14CKRFFano,Edwards15ScalBoundCKRF} originated from \cite{CampanaGuenanciaPaun13ConeSingHolTensor}.
We denote $\rhoe \deq \chi(\st ,\varepsilon^2)$, where
\begin{align*}
    \chi(u,\varepsilon^2) \deq \beta\int_{0}^{u} \dfrac{(r+\varepsilon^2)^\beta - \varepsilon^{2\beta} }{r}dr.
\end{align*}
Then, $\rhoe$ are smooth functions on $X$ and converge to $\sbb $ in $C^\infty_{\mathrm{loc}}(X\setminus D)$ as $\varepsilon \rightarrow 0$. 
We define reference  smooth K\"ahler forms $\omtilet$ by
\begin{align}\label{omtilet}
    \omtilet
    &\deq \omht + k \ddb \rhoe 
    =a_t \omtilez +(1-a_t)\omtileT .
\end{align}
We prove that if we take $k$ sufficiently small, $\omtilet $ is positive for all $t\in \zt$.
Let \const{y}$\refconst{y}>0$ be a constant satisfying
\begin{align}\label{rhy}
	-\refconst{y} \omZ \le \ii R_{h_Y}\le \refconst{y} \omZ \son Y,
\end{align} 
where $ R_{h_Y}$ is the Chern curvature of $h_Y$. Since $h=f^\ast h_Y$ and $\omhT = f^\ast \omZ$, we have
\begin{align}\label{rh}
	-\refconst{y} \omhT \le \ii R_h \le \refconst{y} \omhT \son X.
\end{align}
Let \const{s}$\refconst{s}>0$ and \const{z}$\refconst{z}>1$ be constants such that
\begin{align}
\label{supsy}	\sup_Y |s_Y|_{h_Y} &\le \refconst{s}, \\
\label{omhT}	\omhT = f^\ast \omZ &\le \refconst{z}\omz \son X . 
\end{align}
By (\refs{supsy}), there exists a constant \const{rho}$\refconst{rho}>0$ independent of $\varepsilon$ such that
\begin{align}\label{urho}
    0 \le \rhoe \le \refconst{rho}  \son X.
\end{align}
By the computation in \cite[Section 3]{CampanaGuenanciaPaun13ConeSingHolTensor}, we have 
\begin{align}\label{ddbrho}
    \ddb \rhoe 
  	  &=\beta^2 \dfrac{\ii \langle \nabla s \wedge \nabla s \rangle_h}{\sobe} - \beta \bigl(\sbe -\varepsilon^{2\beta}\bigr)\ii R_h\\
    \notag 	
  	  &\ge -\beta \refconst{y}\refconst{s}^{2\beta} \omhT,
\end{align}
where $\nabla $ is the Chern connection of the line bundle $(\mathcal{O}_X(D),h)$, $R_h$ is its Chern curvature, and $\ii \langle \nabla s \wedge \nabla s\rangle_h$ is a semi-positive closed real $(1,1)$-form combining the wedge product of differential forms with the  Hermitian metric $h$ on $\mathcal{O}_X(D)$.
By (\refs{omtilet}), (\refs{ddbrho}), and (\refs{omhT}), we obtain the following inequalities: 
\begin{align}
    \omtileT
	    &= \omhT +k \ddb \rhoe
		    \ge (1-k\beta \refconst{y}\refconst{s}^{2\beta } )\omhT
			\ge  (1-k\beta \refconst{y}\refconst{s}^{2\beta }\refconst{z} )\omhT,\\
\label{omtilzlower}	\omtilez
	    &=\omz +k \ddb \rhoe 
		    \, \ge \omz -k\beta \refconst{y}\refconst{s}^{2\beta }\omhT
			\> \ge  (1-k\beta \refconst{y}\refconst{s}^{2\beta }\refconst{z} )\omz.
\end{align}
Finally, these inequalities give the positivity of $\omtilet$ for any $t\in \zt$: 
\begin{align*}
\omtilet
	&=\omht + k \ddb \rhoe 
	=a_t \omtilez +(1-a_t ) \omtileT
	\ge  (1-k\beta \refconst{y}\refconst{s}^{2\beta }\refconst{z} )\omht >0.
\end{align*}

By using these reference smooth \kahler \ forms, we consider the following flow of potentials:
\begin{align}\label{peteq}
	\begin{cases}
		\ \dt \pet 
			&=\log \dfrac{(\omtilet +\ddb \pet)^n}{\Omega } - \pet 
				+(1-\beta )\log(\ste) -k\rhoe, \\[7pt]
		\ \pet|_{t=0}
			&=0,
	\end{cases}
\end{align}
where $\Omega $ is a smooth volume form on $X$ satisfying
\begin{align*}
	-\ricOm +(1-\beta )\ii R_h &= \omhinf \ \in \ckd.
\end{align*}
We set $\omet$ by 
\begin{align}\label{ometdef}
	\omet \deq \omtilet +\ddb \pet.
\end{align}
Then, $\omet$ satisfies the following twisted \krf: 
\begin{align}\label{ometeq}
	\begin{cases}
		\ \dt \omet &= - \ricet - \omet + \ete, \\[7pt]
		\ \omt |_{t=0} &= \omtilez(\deq \omz +k\ddb \rhoe),
	\end{cases}
\end{align}
where $\ete \deq (1-\beta)\ddb \log (\ste )+ (1-\beta )\ii R_h$. 
We remark that $\ete$ converges to $2\pi (1-\beta )[D]$  in $C^\infty_{\mathrm{loc}}(X\setminus D)$ and as current on $X$ when $\varepsilon$ goes to $0$.

The validity of these approximations (\refs{peteq}), (\refs{ometeq}) is justified by the following theorem due to Shen~\cite{Shen14CKRF}. 
\begin{theorem}[{\cite[Section 2]{Shen14CKRF}}]\label{Shenthm}
	There exists a subsequence $\varepsilon_i$ converging to $0$ as $i\rightarrow \infty$ such that 
	$\omega_{\varepsilon_i,t}$ converges to $\omega_t$ in $C^\infty_{\mathrm{loc}}(X\setminus D)$ and as current on $X$.
\end{theorem}
Thanks to this theorem, we only need to estimate $\pet$ and $\omet$.

\section{Overview of the proof of Theorem \refs{thmA}}\label{outlinesect}
	
In this section, we outline the proof of Theorem \refs{thmA}.  
First, we need the following formulas.
\begin{proposition}\label{retprop}
The Ricci curvature $\ricet$ and the scalar curvature $\ret$ satisfy the following formulas: 
		\begin{enumerate}
			\item [(a)]
			{
				\abovedisplayskip=-12pt
				\begin{align*}
					\oett\left(\ricet -\ete \right)
						= -\ddb \vet + \ett \omet -\omhT ,
				\end{align*}
			}
			\item [(b)]
						{
							\abovedisplayskip=-12pt
				\begin{align*}
					\oett (\ret - \trete)
						= -\lapet \vet + n\ett -\tromhT,
				\end{align*}
			}
		\end{enumerate}
%	\vspace{-8mm}
	where $\vet \deq \oett \petd +\pet + k \rhoe.$
\end{proposition}
	
\begin{proof}
		(b) follows from (a) by taking traces. We prove (a). 
		By  (\refs{ometeq}), (\refs{omtilet}), and (\refs{ometdef}), we have
\begin{align*}
	\ricet - \ete
		&= - \dt \omet -\omet\\
		&= - \left(\dt \omht + \dt \ddb \pet \right) - \left(\omht + k \ddb \rhoe +\ddb \pet \right)\\
		&= -\ddb (\petd + \pet +k\rhoe) -\left(\omht + \dt \omht \right).
	\end{align*}	
On the other hand, we get
\begin{align*}
	-e^{t-T} (\ricet -\ete )
		&=-e^{t-T}\left(  - \dt \omet -\omet\right)\\
		&=e^{t-T}\left(\dt \omht + \dt \ddb \pet \right)+e^{t-T}\omet\\
		&= \ddb\left( e^{t-T} \petd \right)+e^{t-T}\omet + e^{t-T}\dt \omht.
\end{align*}
Combining these equalities and (\refs{omht}), we obtain (a).  
\end{proof}

By this proposition, to obtain the upper bound for the scalar curvature $R(\omet)$, we only need to estimate $\uet \deq \tromhT$ and $\lapet \vet$. 
We divide our argument into the following 5 steps:
\begin{itemize}
\setlength{\leftskip}{16pt}
		\item[\textbf{Step 1.}] The $C^0$-estimate for $\vet$ (Section \refs{vetsect}).
		\item[\textbf{Step 2.}] The $C^0$-estimate for $\uet \deq \tromhT$ using Step 1 and the parabolic Schwarz lemma (Section \refs{uetsect}).
		\item[\textbf{Step 3.}] The gradient estimate for $\vet$ (Section \refs{gvetsect}).
		\item[\textbf{Step 4.}] The Laplacian estimate for $\vet$ (Section \refs{lapetsect}):
			\begin{align*}
				\lapet \vet \ge - \dfrac{C}{T-t}.
			\end{align*}
		\item[\textbf{Step 5.}] Proof of Theorem \refs{thmA} (Section \refs{lapetsect}).  
\end{itemize}

\section{The \texorpdfstring{$C^0$}{C0}-estimate for \texorpdfstring{$v_{\varepsilon,t}$}{v}}\label{vetsect}
	
In this section, we prove the $C^0$-estimates for $\vet $. More precisely, we prove the following proposition.\const{v}
\begin{proposition}\label{bdvetprop}
		There exists a constant $\refconst{v}>0$ independent of $\varepsilon$ and $t$ such that
			\begin{align*}
				\| \vet \|_{C^0}\le \refconst{v}
			\end{align*}
		holds.
\end{proposition} 
	
To apply the maximum principle, we need the following lemmas.
\begin{lemma}\label{vevlem}
$\vet$ satisfies the following evolution equation
	\begin{align*}
		\dalet \vet = -n + \uet ,
	\end{align*}
where $\uet \deq \tromhT$.
\end{lemma}

\begin{proof}
Differentiating (\refs{peteq}) with respect to $t$, we have
\begin{align}\label{bb}
\notag 
	\dt \petd 
		&= \tret\left(\dt \bigl(\omtilet + \ddb \pet \bigr)\right) - \petd , \\ 
\mbox{i.e. }
	\dt \left( \petd + \pet \right) 
		&= \tret \left(\dt \omht \right)+\lapet \petd .
\end{align}
On the other hand, by (\refs{ometdef}) and (\refs{omtilet}), we have
\begin{align}\label{bc}
		\lapet \pet 
			= \tret (\omet - \omtilet )
			= n - \tret(\omht) - \lapet (k\rhoe).
\end{align}
Combing (\refs{bb}) and (\refs{bc}), we obtain
\begin{align}\label{cc}
	\dt \left( \petd + \pet + k \rhoe\right ) 
		&=\lapet ( \petd + \pet +k\rhoe )-n + \tret \left(\omht + \dt \omht \right).
\end{align}
		
Next, by using (\refs{bb}) again, we have
\begin{align}\label{dd}
	\dt \left( -\ett \petd  \right)
		&=-\ett \dt ( \petd + \pet )\\
	\notag 
		&=-\tret \left(\ett \dt \omht \right)-\lapet (\ett \petd ).
\end{align}
By (\refs{cc}), (\refs{dd}), and (\refs{omht}), we get the assertion.
\end{proof}
	
\begin{lemma}\label{vollem}
There exists a constant $\refconst{vol} >1$ independent of $\varepsilon$ ant $t$ satisfying the following inequalities: $\const{vol}$ 
\begin{enumerate}
	\item[(a)] {\abovedisplayskip=-12pt
				\begin{align*}
					\dfrac{1}{\refconst{vol}}\dfrac{\Omega}{\sobe }
						&\le \omtilez^n 
						\le \refconst{vol}\dfrac{\Omega}{\sobe }.
				\end{align*}
}
	\item[(b)] {\abovedisplayskip=-12pt
					\begin{align*}
						\omtilet^n  \le \refconst{z}^n \refconst{vol}\dfrac{\Omega}{\sobe }.
					\end{align*}
}
\end{enumerate}
\end{lemma}
	
\begin{proof}
	The first inequality follows from (\refs{ddbrho}). We prove (b).
	For $0<k<\refconst{z}$, by (\refs{omtilet}) and (\refs{omhT}), we have	
\begin{align*}
	\omtileT
		&= \omhT + k \ddb \rhoe
		\le \refconst{z}\omz + k \ddb\rhoe
		\le \refconst{z}\omtilez .
\end{align*}
Since $\refconst{z}>1$, we have
\begin{align*}
	\omtilet
		&=a_t \omtilez +(1-a_t)\omtileT
		\le a_t \omtilez +\refconst{z}(1-a_t)\omtilez
		\le \refconst{z}\omtilez.
\end{align*}
Therefore we get the assertion.
\end{proof}

Using these lemmas, we can prove the uniform lower boundedness of $\vet$.
\begin{proposition}\label{vlprop}
$\vet$ is uniformly lower bounded. More precisely, there exists a constant $\const{vl}\refconst{vl}>0$ independent of $\varepsilon$ and $t$ such that
\begin{align*}
		\vet \ge -\refconst{vl}.
\end{align*}
\end{proposition}
	
\begin{proof}
By Lemma \refs{vevlem} and the semi-positivity of $\omhT$, we have
\begin{align*}
	\dalet \left(\vet +nt\right)  
		=  \uet 
		= \tromhT \ge 0.
\end{align*}
Thus, the maximum principle for $\vet +nt $ gives the following:
\begin{align*}
	\vet +nt 
		\ge \min_{X\times \{0\}} (\vet +nt) 
		= (1-e^{-T})\dot{\varphi}_{\varepsilon,0}+ k \rhoe \ge (1-e^{-T})\dot{\varphi}_{\varepsilon,0}. 
\end{align*}
Lemma \refs{vollem} (a) and (\refs{urho}) give the lower boundedness of right hand side as follows:
\begin{align*}
	\dot{\varphi}_{\varepsilon,0}
		&= \log \dfrac{\omtilez^n}{\Omega/\sobe }-\varphi_{\varepsilon,0} -k\rhoe
		\ge -\log \refconst{vol} -k\refconst{rho}.
\end{align*}
Therefore we get the assertion. 
\end{proof}

The upper bound for $\vet$ follows from the next proposition.
\begin{proposition}\label{petprop}
We have the following inequalities: 
\begin{itemize}
	\item[(a)] $\pet \le \refconst{pet}$,
	\item[(b)] $\petd \le \refconst{petd}$,
\end{itemize}
where $\refconst{pet}>0$, $\refconst{petd}>0$ independent of $\varepsilon$ and $t$.\const{pet}
\const{petd}
\end{proposition}
	
\begin{proof}
(a) Since $\varphi_{\varepsilon,0}=0$, $\pet $ takes maximum at $(x_0,t_0) \in X \times (0,T)$. 
By (\refs{peteq}) and Lemma \refs{vollem} (b), we have the following inequality at $(x_0,t_0) $:
\begin{align*}
	0 \le \dt \pet 
		&\le \log \dfrac{\omtilet^n}{\Omega/\sobe }- \pet -k\rhoe
		\le \log (\refconst{z}^n \refconst{vol}) -\pet .
\end{align*}
We obtain 
$\pet(x_0,t_0) \le \log (\refconst{z}^n \refconst{vol}) \ \ddeq \ \refconst{pet}.$
Since $(x_0,t_0) $ is arbitrary, $\pet \le \refconst{pet}$ holds on $X\times [0,T)$.
	
(b) We set $H_{\varepsilon,t}\deq (1-e^t) \petd + \pet + k\rhoe +nt$. The same computation in Lemma \refs{vevlem} gives
\begin{align*}
	\dalet H_{\varepsilon,t}= \tret(\omz )>0.
\end{align*}
By the maximum principle for $H_{\varepsilon,t}$, we have
\begin{align*}
	H_{\varepsilon,t} \ge \min_{X\times \{0\}} H_{\varepsilon,t}=k\rhoe \ge 0.
\end{align*}
Therefore, combining with (a) and (\refs{urho}), we get the upper bound for $\petd$:
\begin{align*}
	\petd \le \dfrac{\pet +k\rhoe +nt}{e^t-1}\le \dfrac{\refconst{pet} + k \refconst{rho} +nT}{e^t-1}.
\end{align*}
		Combining with the uniform local estimate for the parabolic equation, we get the assertion.
\end{proof}
	
\section{The \texorpdfstring{$C^0$}{C0}-estimate for \texorpdfstring{$u_{\varepsilon,t}$}{u}}\label{uetsect}
	
In this section, we prove the following proposition.
\begin{proposition}\label{uetprop}
	There exists a constant $\refconst{uet}>0$ independent of $\varepsilon$ and $t$ such that$\const{uet}$
\begin{align*}
	0 \le \uet \deq \tromhT \le \refconst{uet}.
\end{align*}
\end{proposition}
	
To prove this, we need the estimate $\ete$ and the parabolic Schwarz lemma. 
\begin{lemma}\label{etelem}
We have the following inequalities of $\ete$.
	\begin{enumerate}
		\item[(a)] Lower boundedness of $\ete$:
			\begin{align*}
				\ete 
					&= (1-\beta )\dfrac{\varepsilon^2}{\ste} 
							\left( 
									\dfrac{\ii \langle \nabla s \wedge \nabla s \rangle_h}{\ste} 
									+ \ii R_h 
							\right) \\
					&\ge -(1-\beta )\refconst{y}\omhT.
			\end{align*}
		\item[(b)] For any K\"ahler form $\omega$, we have 
			\begin{align*}
				-\langle \ete, \omhT \rangle_\omega 
						&\le (1-\beta) \refconst{y} |\omhT|_\omega^2
						\le (1-\beta )\refconst{y} (\mathrm{tr}_\omega (\omhT) )^2.
			\end{align*}
	\end{enumerate}
\end{lemma}

By the fact that $\omhT$ is the pullback of  the \kahler \ form $\omega_Z$ by the holomorphic map $f$, 
we can use the parabolic Schwarz lemma which is obtained by Song-Tian~\cite{SongTian07KRFSurfKod}. 
This is the parabolic version of \cite{Yau78Schwarz}. 
Combining it with Lemma \refs{etelem} (b), we have the following Lemma.

\begin{lemma}[parabolic Schwarz lemma]\label{paraboliclem}
	$\uet$ and $\log \uet $ satisfy the following inequalities.\const{para}
\begin{enumerate}
		\item[(a)] {\abovedisplayskip=-12pt
			\begin{align*}
				\lapet \uet 
					&\ge -C_Z \uet^2 +\langle \ricet ,\omhT\rangle _{\omet }\\
					&\ge -\refconst{para}  \uet^2 +\langle \ricet -\ete ,\omhT\rangle _{\omet }.
			\end{align*}
		}		
		\item[(b)] {\abovedisplayskip=-12pt
			\begin{align*}
			\dalet \uet 
				&\le \uet + C_Z \uet^2 - \langle \ete, \omhT\rangle_{\omet} - \dfrac{\nuets}{\uet }\\
				&\le \uet +\refconst{para} \uet^2 - \dfrac{\nuets}{\uet }.
			\end{align*}
		}
		\item[(c)] {\abovedisplayskip=-12pt
		\begin{align*}
			\dalet \log \uet 
				&\le C_Z \uet +1 - \dfrac{\langle \ete, \omhT\rangle_{\omet} }{\uet}\\
				&\le \refconst{para} \uet+1.
		\end{align*}
}
\end{enumerate}
	Here, $\nabla $ is $(1,0)$-part of the Levi-Civita connection of $\omet$, 
	$C_Z>0$ is an upper bound for the bisectional curvature of $\omega_Z$, and $\refconst{para}\deq C_Z + (1-\beta)\refconst{y}>0$.  
\end{lemma}

\noindent 
\textit{Proof of Proposition \refs{uetprop}}.
We set \const{parap}$G_{\varepsilon,t} \deq \log \uet -\refconst{parap}\vet$, where $\refconst{parap}\deq \refconst{para}+1>0$.  
The uniform upper boundedness of $G_{\varepsilon,0}$ follows from (\refs{omhT}), (\refs{omtilzlower}) and Proposition \refs{bdvetprop}. 
We assume that $G_{\varepsilon,t} $ achieves maximum at $(x_0,t_0) \in X \times (0,T)$.
At this point, by Lemma \refs{paraboliclem} (c) and Lemma \refs{vevlem}, we have $\uet >0$ and
\begin{align*}
	\dalet G_{\varepsilon,t} 
		&\le (\refconst{para} \uet+1) - \refconst{parap} (\uet -n) 
		= -\uet +(\refconst{parap}n+1) .
\end{align*}
By using the uniform boundedness of $\vet$ (Proposition \refs{bdvetprop}), we obtain \const{gub}
\begin{align*}
	G_{\varepsilon,t}(x_0,t_0) 
		\le \log (\refconst{parap}n+1) -\refconst{parap}\vet 
		\le \log (\refconst{parap}n+1) +\refconst{parap}\refconst{v} .
\end{align*}
Since $(x_0,t_0)$ is arbitrary, we have $G_{\varepsilon,t}\le \refconst{gub}$ on $X\times [0,T)$. 
Hence, by using Proposition \refs{bdvetprop} again, we get the assertion.

\section{The Gradient estimate for \texorpdfstring{$v_{\varepsilon,t}$}{v}}\label{gvetsect}

In this section, we prove the following gradient estimate for $\vet$.

\begin{proposition}\label{nvetsprop}
	There exists a uniform constant $\refconst{nvets}>0$ which is independent of $\varepsilon$ and $t$ such that \const{nvets}
\begin{align*}
	\nvets \le \refconst{nvets}.
\end{align*}
\end{proposition}

To prove this proposition, as in \cite{Zhang10ScalFTSingKRF}, we set $\Psiet \deq \dfrac{\nvets}{A-\vet}$ 
and use the maximum principle to $\Psiet +\uet$. 
Here $A>\refconst{v}+1$ is a fixed constant. 
\begin{lemma}\label{psietlem}
	We have the following formulas.
	\begin{enumerate}
		\item[(a)]  {\abovedisplayskip=-12pt
		\begin{align*}
			&\dalet \nvets \\
			&= \nvets  - \ete ( \nvet , \nbvet )
				+ 2 \re\langle \nvet, \nuet \rangle_{\omet}
			 \\&\quad 
			 - \nnvets - \nnbvets  .
		\end{align*}
	}
		\item[(b)] {\abovedisplayskip=-12pt
		\begin{align*}
			&\dalet \lapet \vet \\
			&= \lapet \vet + \lapet \uet
				 + \langle \ricet - \ete , \ddb\vet\rangle_{\omet}.
		\end{align*}
	}
		\item[(c)] {\abovedisplayskip=-12pt
		\begin{align*}
		&\left(\dt - \lapet  \right) \Psiet \\
		&= \dfrac{1}{A-\vet}\biggl( \nvets  - \nnvets - \nnbvets- \ete ( \nvet , \nbvet ) \\
			&\quad  + 2\re\langle \nvet, \nuet \rangle_{\omet}\biggr)\\
			&\quad + \dfrac{1}{(A-\vet)^2}\left((\uet - n)\nvets
						-2\re \langle \nabla \nvets , \nvet \rangle_{\omet}\right)\\
			&\quad - \dfrac{2}{(A-\vet)^3}\nvetf .
		\end{align*}
	}
	\end{enumerate}
\end{lemma}

\noindent
\textit{Proof of Proposition \refs{nvetsprop}}. 
We will apply the maximum principle to $\Psiet +\uet$. First, we estimate $\dalet \Psiet$. 
By Lemma \refs{etelem} (a) and Proposition \refs{uetprop}, we have
\begin{align}\label{AA}
	- \ete ( \nvet , \nbvet )	
		&\le (1-\beta)\refconst{y} \omhT  ( \nvet , \nbvet )	\\
\notag 
		&\le  (1-\beta)\refconst{y} \tromhT  \ \omet ( \nvet , \nbvet )	\\
\notag 
		&\le (1-\beta)\refconst{y} \refconst{uet} \nvets.
\end{align}

For sufficiently small constant $\delta >0$ which will be determined later, we have 
\begin{align}\label{BB}
	2\re \langle \nvet, \nuet \rangle_{\omet}
		&\le 2 |\nabla\vet |_{\omet} |\nabla \uet |_{\omet}
		\le \dfrac{1}{\delta}\nvets + \delta \nuets.
\end{align}

Since  
\begin{align}\label{npsiet}
	\nabla \Psiet = \dfrac{\nabla \nvets}{A-\vet}+ \dfrac{\nvets}{(A-\vet)^2}\nvet ,
\end{align}
we have
\begin{align}\label{CC}
	& -\dfrac{2-\delta }{(A-\vet)^2}\re \langle \nabla \nvets , \nvet \rangle_{\omet}\\
\notag 
	&=-\dfrac{2-\delta}{A-\vet }\re \langle \nabla \Psiet , \nvet \rangle_{\omet}+(2-\delta )\dfrac{\nvetf}{(A-\vet )^3}.
\end{align}

On the other hand, the Cauchy-Schwarz inequality gives
\begin{align*}
	&|\langle \nabla \nvets , \nvet \rangle_{\omet}|\\
	&=\left|\guijb \guklb \Bigl((\delk \deli \vet )( \deljb \vet )( \dellb \vet ) + ( \deli \vet )( \delk \deljb \vet )( \dellb \vet) \Bigr) \right|\\
	&\le \nvets (\nnvet +\nnbvet )\\
	&\le \sqrt{2}\nvets (\nnvets +\nnbvets)^{1/2}.
\end{align*}
By using this inequality, we obtain the following:
\begin{align}\label{DD}
	&\dfrac{-\delta }{(A-\vet)^2}\re \langle \nabla \nvets , \nvet \rangle_{\omet}\\
\notag 
		&\le \dfrac{\delta }{(A-\vet)^2}\left(\sqrt{2} \nvets (\nnvets + \nnbvets)^{1/2} \right)\\
\notag 
		&=\sqrt{2} \delta \dfrac{\nvets}{(A-\vet)^{3/2}} \dfrac{ (\nnvets + \nnbvets)^{1/2}}{(A-\vet)^{1/2}}\\
\notag 
		&\le \dfrac{\delta}{2}\dfrac{\nvetf}{(A-\vet)^{3}}+ \delta \dfrac{ \nnvets + \nnbvets}{A-\vet}.
\end{align}

Therefore, combining Proposition \refs{psietlem} (c) with (\refs{AA}),(\refs{BB}), (\refs{CC}), \refs{DD}, Proposition \refs{uetprop}, Proposition \refs{bdvetprop}, and $A- \refconst{v}>1$, we obtain the following inequality:\const{Psiet}
\begin{align*}
		&\dalet \Psiet\\
		&\le \refconst{Psiet} \nvets 
			+\delta \nuets -\dfrac{2-\delta}{A-\vet}\re \langle \nabla \Psiet , \nvet \rangle_{\omet}
			 -\dfrac{\delta}{2}\dfrac{\nvetf}{(A+\refconst{v})^3},
\end{align*}
where $\refconst{Psiet} \deq 1+(1-\beta) \refconst{y}\refconst{uet}+(1/\delta) +\refconst{uet}>0$.\const{uett}

On the other hand, by Lemma \refs{paraboliclem} (a) and Proposition \refs{uetprop}, we have
\begin{align*}
	\dalet \uet 
		&\le  \uet +\refconst{para} \uet^2 - \dfrac{\nuets}{\uet }
		\le \refconst{uett}-2\delta \nuets,
\end{align*}
where $\refconst{uett}\deq \refconst{uet} +\refconst{para} \refconst{uet}^2$ and $0<\delta< 1/(2\refconst{uet})$. 

Finally, we obtain the following inequality:
\begin{align}\label{EE}
	&\dalet (\Psiet+\uet)\\ 
\notag 
	&=\refconst{uett}+\refconst{Psiet} \nvets
		 - \delta \nuets 
	\\\notag &\quad 
		 - \dfrac{2-\delta}{A-\vet}\re \langle \nabla \Psiet , \nvet \rangle_{\omet} 
		 -\dfrac{\delta}{2}\dfrac{\nvetf}{(A+\refconst{v})^3}\\
\notag 
	&\le \refconst{uett}+\left( \refconst{Psiet}+\dfrac{1}{\delta} \right) \nvets\\
\notag 
	&\quad 
		 -\dfrac{2-\delta}{A-\vet} \re \langle \nabla (\Psiet +\uet) , \nvet \rangle_{\omet}
		-\dfrac{\delta}{2}\dfrac{\nvetf}{(A+\refconst{v})^3}.
\end{align}
Here, we used the following inequality:
\begin{align*}
	\dfrac{2-\delta}{A-\vet}\re \langle \nuet , \nvet \rangle_{\omet}
		&\le 2 |\nvet |_{\omet} |\nuet |_{\omet}\\
		&\le \dfrac{1}{\delta}\nvets  +\delta \nuets.
\end{align*}

The uniform boundedness of $\Psi_{\varepsilon,0}+ u_{\varepsilon,0}$ follows from \cite[Section 4]{CampanaGuenanciaPaun13ConeSingHolTensor}, Proposition \refs{bdvetprop} and Proposition \refs{uetprop}. If $\Psiet +\uet $ achieves maximum at $(x_0,t_0) \in X\times (0,T)$, by (\refs{EE}), we have the following estimate:
\begin{align*}
0
&\le \refconst{uett}+\left( \refconst{Psiet}+\dfrac{1}{\delta} \right) \nvets-\dfrac{\delta}{2}\dfrac{1}{(A+\refconst{v})^3}\nvetf \sat (x_0,t_0).
\end{align*}
It follows that there exists a constant $\refconst{nvetsatpt}>0$ satisfying\const{nvetsatpt}
\begin{align*}
\nvets \le \refconst{nvetsatpt} \sat (x_0,t_0),
\end{align*}
which does not depend on $\varepsilon$ and $t$. By using the definition of $\Psiet$, $A-\vet>1$, and Proposition \refs{uetprop}, we have the uniform upper bound for $\Psiet +\uet$ on $X \times \zt$. 
Therefore, we obtain the uniform upper bound for $\nvets$.

\section{The Laplacian estimate for \texorpdfstring{$v_{\varepsilon,t}$}{v}}\label{lapetsect}
	
In this section, we estimate $\lapet \vet $. 
In order to prove the uniform upper boundedness of $\lapet \vet$, 
we need the lower bound for the scalar curvature due to Edwards  \cite[Corollary 4.3]{Edwards15ScalBoundCKRF}.
\begin{proposition}[{\cite[Corollary 4.3]{Edwards15ScalBoundCKRF}}]\label{retlower}
The scalar curvature $R(\omet)$
 is uniformly bounded from below by\const{lower}
	\begin{align*}
		R(\omet) -\trete \ge -\refconst{lower},
	\end{align*}
	where $\refconst{lower}>0$ is a constant independent of  $\varepsilon$ and $t$.
\end{proposition}
	
Using this estimate, we can easily obtain the following upper bound. 
\begin{proposition}
	There exists a uniform constant $\refconst{ulapvet}>0$ which is independent of $\varepsilon$ and $t$ such that\const{ulapvet}
		\begin{align*}
		\lapet \vet \le \refconst{ulapvet}.
		\end{align*}
\end{proposition}
	
\begin{proof}
By Proposition \refs{retprop}, Proposition \refs{retlower}, and $\uet \ge 0$, we have 
\begin{align*}
	\lapet \vet
		&=n\ett -\uet -\oett(\ret -\trete)
		\le n + \refconst{lower} \, \ddeq \, \refconst{ulapvet},
\end{align*}
which proves the assertion.
\end{proof}

\begin{proposition}\label{llapvetprop}
There exists a constant $\refconst{llapvet}>0$ independent of $\varepsilon$ and $t$ such that \const{llapvet}
	\begin{align*}
		\lapet \vet \ge -\dfrac{\refconst{llapvet}}{T-t}.
	\end{align*}
\end{proposition}
	
\begin{proof}
As in \cite[Section 3.3]{Zhang10ScalFTSingKRF}, we set 
\begin{align*}
\Phiet \deq \dfrac{B-\lapet \vet}{B-\vet},
\end{align*}
where $B>0$ is a sufficiently large uniform constant satisfying $B-\refconst{ulapvet}>0$, and $B-\refconst{v}>1$ so that the numerator and the denominator of $\Phiet$ are positive.
Straightforward calculations show that 
\begin{align}
\label{AAAA}	
	&\biggl(\dt - \lapet  \biggr) \Phiet \\
\label{aaa}	
	&= \dfrac{-1}{B-\vet}\lapet \vet +\dfrac{1}{ (B-\vet)^2 }(\uet -n )(B - \lapet \vet )\\
\label{bbb}	
	&\quad -\dfrac{1}{B-\vet}\Bigl(\langle \ricet -\ete , \ddb \vet \rangle_{\omet}+\lapet  \uet \Bigr)\\
\notag 
	&\quad -\dfrac{2}{ B-\vet } \re \langle \nabla \Phiet , \nvet \rangle_{\omet} .
\end{align}

We first estimate  (\refs{aaa}). By using $B-\vet >1$, $B-\lapet \vet>0$, and Proposition \refs{uetprop},  we have \const{Bvet}
\begin{align}\label{AAA}
	&\dfrac{-1}{B-\vet}\lapet \vet +\dfrac{1}{(B-\vet)^2}(\uet -n )(B-\lapet \vet)\\
\notag 
	&=\left( \dfrac{B-\lapet \vet}{B-\vet }+\dfrac{-B}{B- \vet }  \right)+ \dfrac{1}{(B-\vet)^2}(\uet -n )(B-\lapet \vet)\\
\notag 
	&\le \refconst{Bvet}(B-\lapet \vet),
\end{align}
where $\refconst{Bvet}\deq 1+\refconst{uet}>0$.

We next estimate (\refs{bbb}). By using Lemma \refs{paraboliclem} (a) and Proposition \refs{retprop} (a), we obtain \const{lapuet}
\begin{align*}
	&-\langle \ricet -\ete , \ddb \vet \rangle_{\omet} -\lapet \uet\\
	&\le - \langle \ricet -\ete , \ddb \vet + \omhT \rangle_{\omet} + \refconst{lapuet}\\
	&=\dfrac{1}{1-\ett }| \ddb \vet +\omhT|_{\omet}^2 -\dfrac{\ett}{1-\ett}(\lapet \vet + \uet ) + \refconst{lapuet},
\end{align*}
where $\refconst{lapuet}\deq \refconst{para} \refconst{uet}^2>0$.
By using Proposition \refs{uetprop}, the first term is estimated as follows:
\begin{align*}
	| \ddb \vet +\omhT|_{\omet}^2 
	&\le (1+\delta )\nnbvets + (1+1/\delta )|\omhT|_{\omet}^2 \\
	& \le (1+\delta )\nnbvets +\refconst{Constant} ,
\end{align*}
where $\delta >0$ is a uniform constant determined later and $\refconst{Constant} \deq  \left(1+1/\delta\right)\refconst{uet}^2 >0$. 
For the second term, we have
\begin{align*}
	 -\dfrac{\ett}{1-\ett}(\lapet \vet + \uet )
		 &=\dfrac{\ett}{1-\ett }(B-\lapet \vet)-\dfrac{B\ett}{1-\ett}\\
		 &\le \dfrac{1}{1-\ett }(B-\lapet \vet).
\end{align*}
Finally, we get\const{Constant}
\begin{align}\label{BBB}
	&-\dfrac{1}{B-\vet}\left(\langle \ricet -\ete , \ddb \vet \rangle_{\omet}+\lapet  \uet \right)\\
\notag 
	& \le \dfrac{C_T}{T-t}\left( \dfrac{1+\delta}{B-\vet}\nnbvets 
		+ \refconst{Constant}\right)
		+ \dfrac{C_T}{T-t}(B-\lapet \vet)+\refconst{lapuet},
\end{align}
where $C_T>0$ is a uniform constant satisfying
\begin{align*}
	\dfrac{1}{1-\ett}\le \dfrac{C_T}{T-t}
\end{align*}
for all $t \in \zt$.

Combining (\refs{AAAA}) with (\refs{AAA}) and (\refs{BBB}), we get\const{const}
\begin{align}\label{Estimate1}
\notag 
	&\dalet \Phiet\\
\notag 
	&\le \refconst{Bvet}(B-\lapet \vet)\\
\notag 
		&\quad +\dfrac{C_T}{T-t}\left( \dfrac{1+\delta}{B-\vet}\nnbvets + \refconst{Constant} \right) + \dfrac{C_T}{T-t}(B-\lapet \vet)+\refconst{lapuet}\\
\notag 
		&\quad -\dfrac{2}{B-\vet }\re \langle \nabla \Phiet , \nabla \vet \rangle_{\omet }\\
\notag 
	&\le \dfrac{\refconst{const} }{T-t}+\dfrac{\refconst{const}}{T-t}(B-\lapet \vet)+\dfrac{C_T}{T-t}\dfrac{1+\delta}{B-\vet}\nnbvets\\
\notag 
	&\quad -\dfrac{2}{B-\vet }\re \langle \nabla \Phiet , \nabla \vet \rangle_{\omet }, \\
	&\dalet (T-t)\Phiet=-\Phiet +(T-t)\dalet \Phiet\\
\notag 	
	&\le (T-t)\dalet \Phiet\\
\notag 
	&\le \refconst{const} + \refconst{const}(B-\lapet \vet)+C_T\dfrac{1+\delta}{B-\vet}\nnbvets\\
\notag 
	&\quad  -\dfrac{2}{B-\vet }\re \langle \nabla \left( (T-t)\Phiet\right)  , \nabla \vet \rangle_{\omet }.
\end{align}

We set $\tilPsiet \deq \dfrac{\nvets}{B-\vet}$. 
Combining Lemma \refs{psietlem} (c) with (\refs{npsiet}), (\refs{AA}), Proposition \refs{uetprop} and Proposition \refs{nvetsprop}, we have\const{tilpsiconst}
\begin{align}\label{Estimate2}
	&\dalet \tilPsiet\\
\notag 
	&= \dfrac{1}{B-\vet}\left( \nvets  - \nnvets - \nnbvets- \ete ( \nvet , \nbvet ) \right) \\
\notag 
			&\quad + \dfrac{1}{(B-\vet)^2}(\uet - n)\nvets 
			  - \dfrac{2}{B-\vet} 
				\re\langle \nabla (\tilPsiet - \uet) , \nvet \rangle_{\omet}\\
\notag 
	&\le \refconst{tilpsiconst}
			-\dfrac{\nnbvets}{B-\vet}
			- \dfrac{2}{B-\vet} \re\langle \nabla (\tilPsiet - \uet) , \nvet \rangle_{\omet},
\end{align}
where $\refconst{tilpsiconst}\deq \refconst{nvets}(1+(1-\beta)\refconst{y}\refconst{uet}+\refconst{uet})>0$.

We next estimate $\uet$.
We first note that  the following estimate holds: 
\begin{align}\label{EstimateFor3}
	\dfrac{4}{B-\vet}\re \langle \nabla \uet , \nabla \vet \rangle_{\omet }
		\le \delta \nuets +\dfrac{4}{\delta}\nvets.
\end{align}
By using Lemma \refs{paraboliclem} (b), Proposition \refs{uetprop}, Proposition \refs{nvetsprop} and (\refs{EstimateFor3}), we get \const{lapuett}
\begin{align}\label{Estimate3}
		\dalet \uet 
		&\le \refconst{uet} +\refconst{para}\refconst{uet}^2- \dfrac{1}{\refconst{uet}} \nuets\\
\notag 
		&\le - \dfrac{4}{B-\vet}\re \langle \nabla \uet , \nabla \vet \rangle_{\omet } +\refconst{lapuett},
\end{align}
where we take $0<\delta <1/\refconst{uet}$ and $\refconst{lapuett}\deq \refconst{uet} +\refconst{para}\refconst{uet}^2 +4\refconst{nvets}/\delta>0$.

Combining (\refs{Estimate1}), (\refs{Estimate2}), and (\refs{Estimate3}), we have\const{cconst}
\begin{align*}
&\dalet \left((T-t)\Phiet +2C_T \tilPsiet +2C_T \uet \right)\\
	&\le \refconst{cconst} + \refconst{cconst}(B-\lapet \vet) -C_T\dfrac{1-\delta}{B+\refconst{v}}\nnbvets\\
	&\quad-\dfrac{2}{B-\vet }\re \left\langle \nabla  \left((T-t)\Phiet +2C_T \tilPsiet +2C_T \uet \right) , \nabla \vet \right\rangle_{\omet }.
\end{align*}
The uniform boundedness of $(T-t)\Phiet +2C_T \tilPsiet +2C_T \uet$ at $t=0$ follows from \cite[Section 4]{CampanaGuenanciaPaun13ConeSingHolTensor}, Proposition \refs{bdvetprop}, Proposition \refs{uetprop} and Proposition \refs{nvetsprop}. 
If $(T-t)\Phiet +2C_T \tilPsiet +2C_T \uet$ achieves maximum at $(x_0,t_0) \in X\times (0,T)$, we have the following estimate at this point:
\begin{align*}
	0
	&\le \refconst{cconst} + \refconst{cconst}(B-\lapet \vet) -C_T\dfrac{1-\delta}{B+\refconst{v}}\nnbvets\\
	&\le  \refconst{cconst} + \refconst{cconst}(B-\lapet \vet) -C_T\dfrac{1-\delta}{B+\refconst{v}} \left(\dfrac{1}{n}(B-\lapet \vet)^2-\dfrac{B^2}{n}\right).
\end{align*}
Therefore, at this point, there exists a constant $\refconst{ptlapetvet}>0$ satisfying \const{ptlapetvet}
\begin{align*}
	-\lapet \vet\le \refconst{ptlapetvet} \sat (x_0,t_0)
\end{align*}
which is independent of $\varepsilon$, $t$, and $(x_0,t_0)$. Combining Proposition \refs{bdvetprop}, Proposition \refs{nvetsprop}, and Proposition \refs{uetprop}, we obtain the uniform upper boundedness of $(T-t)\Phiet +2C_T \tilPsiet +2C_T \uet$ on $X \times \zt$. 
Therefore we get the lower bound for  $\lapet \vet$.

\end{proof}
\noindent
\textit{Proof of Theorem \refs{thmA}:} By Proposition \refs{retprop}, and Proposition \refs{llapvetprop}, we have
\begin{align*}
\ret - \trete
	&=\dfrac{1}{1-\ett}\left( -\lapet \vet + n\ett -\uet \right)\\
	&\le \dfrac{C_T}{T-t}\left( \dfrac{\refconst{llapvet}}{T-t} +n\right)
	\le \dfrac{C}{(T-t)^2},
\end{align*}
where $C>0$ does not depend on $\varepsilon$ and $t$. Therefore, by taking $\varepsilon_i\rightarrow 0$, we get the assertion. \hfill $\square$

\vspace{2mm}
\noindent 
\textbf{Acknowledgment. }The author would like to express his gratitude to his supervisor Prof. Shigeharu Takayama for various comments.
This work is supported by the Program for Leading Graduate Schools, MEXT, Japan.

\bibliographystyle{amsalphaurlmod}
\bibliography{reference_modified}
\end{document}